\documentclass[12pt]{amsart}
\usepackage[utf8]{inputenc}
\usepackage{amsmath}
\usepackage{amssymb}
\usepackage{amsfonts}
\usepackage{amsthm}
\usepackage{mathrsfs} 
\usepackage{bm}
\usepackage{bbm}
\usepackage{tikz}
\usepackage{centernot}
\usepackage{hyperref}
\usepackage[autostyle]{csquotes}  

\usepackage[ruled,vlined]{algorithm2e}
\usepackage[margin=1.2in]{geometry}
\newcommand{\N}{\Bbb{N}}

\newcommand{\Z}{\Bbb{Z}}

\newcommand{\Q}{\Bbb{Q}}

\newcommand{\CC}{\mathcal{C}}

\newcommand{\HH}{\mathcal{H}}
\newcommand{\II}{\mathcal{I}}
\newcommand{\JJ}{\mathcal{J}}

\newcommand{\PP}{\mathcal{P}}
\newcommand{\QQ}{\mathcal{Q}}

\newcommand{\TT}{\mathcal{T}}

\hypersetup{
    colorlinks=true,
    linkcolor=blue,
    filecolor=magenta,
    urlcolor=blue,
    citecolor=blue
}

\makeatletter
\newtheorem*{rep@theorem}{\rep@title}
\newcommand{\newreptheorem}[2]{%
\newenvironment{rep#1}[1]{%
 \def\rep@title{#2 \ref{##1}}%
 \begin{rep@theorem}}%
 {\end{rep@theorem}}}
\makeatother

\newtheorem{thm}{Theorem}
\newreptheorem{thm}{Theorem}

\newtheorem{result}{Result}[section]
\newtheorem{lem}[result]{Lemma}
\newtheorem{prp}[result]{Proposition}
\newtheorem{cor}[result]{Corollary}

\newtheorem{clm}[result]{Claim}

\theoremstyle{definition}
\newtheorem{rmk}[result]{Remark}
\newtheorem*{defn}{Definition}

\newtheorem{delusion}{Delusion}

\newtheorem*{ack}{Acknowledgements}

\theoremstyle{remark}

\DeclareMathOperator{\newp}{\bm{\mathsf{newpairs}}}

\newcommand{\hide}[1]{}
\newcommand{\edit}[1]{}%{\color{red}{#1}}}

\newcommand{\rough}[1]{}%\textbf{\textcolor{blue}{#1}}}
\definecolor{darkgreen}{RGB}{75,150,75}
\newcommand{\review}[1]{}%\textcolor{darkgreen}{#1}}
\newcommand{\dc}[1]{}%\textcolor{orange}{dc: #1}}
\newcommand{\zh}[1]{}%\textcolor{blue}{zh: #1}}

\newcommand{\thide}[1]{}%#1}
\newcommand{\pub}[1]{}%\textcolor{purple}{#1}}

\title{On a method of Alweiss}
\author{Zach Hunter}
\email{zachary.hunter@exeter.ox.ac.uk}
\date{\today}

\begin{document}
\begin{abstract}
    Recently, Alweiss settled Hindman's conjecture over the rationals. In this paper, we provide our own exposition of Alweiss' result, and show how to modify his method to also show that sums of distinct products are partition regular over the rationals.
\end{abstract}

\maketitle

\section{Introduction}

Throughout, we abide by the convention that the natural numbers, $\N$, do not contain zero. Additionally, given $n\in \N$, we write $[n]:=\{1,\dots,n\}$.

Very recently Alweiss \cite{alweiss} proved the following. 
\begin{thm} { \cite[Theorem~1]{alweiss} (Hindman's over rationals)}\label{alw result} Fix $r,k<\infty$. For any $r$-coloring $C:\Q_+\to \{1,\dots,r\}$ of the positive rationals, we can find $x_1,\dots,x_k\in \Q_+$ such that
\[\left\{\sum_{i\in I}x_i:\textrm{non-empty } I\subset [k]\right\}\cup \left\{\prod_{j\in J}x_j:\textrm{non-empty }J\subset [k]\right\}\] is monochromatic under $C$.
\end{thm}
\begin{rmk}\label{trivial N}
    Here, we do not require that our $x_i$ be distinct. This makes it important that $0\not \in \Q_+$, else the result is trivial (by taking $x_1=x_2 = \dots = x_k:= 0$). 

    However, this choice does not make the problem any easier. Indeed, if $C$ has the property that $C(q)\neq C(2q)$ for all $q\in \Q_+$ (which can easily be obtained, by at worst doubling the number of colors), then we will unable to take $x_i = x_j$ for distinct $i,j$, as then $C(x_i) \neq C(x_i+x_j)$). So, we merely keep this convention for notational convenience, since it would be unpleasant to address distinctness during our proofs.
\end{rmk}
\noindent In this paper, we give our own exposition of his proof, and go on to obtain a stronger result.

Namely, we extend\footnote{\textbf{Nota bene:} in this version of present manuscript, we do not include a proof of this result. All the necessary ideas will be presented, and a brief sketch is stated in Section~\ref{sketch of new}. A clear argument will be added in a later version. For now we have focused on just documenting how to re-prove Theorem~\ref{alw result}, for clarity of exposition.} his result to a more general set of patterns (sums of `disjoint products'). \begin{thm}\label{rational main}(Generalized Hindman's over rationals) 
    Fix $r,k<\infty$. For any $r$-coloring $C:\Q_+\to \{1,\dots,r\}$, we can find $x_1,\dots,x_k\in \Q_+$ such that:
    \[\left\{\sum_{j=1}^\ell \prod_{i\in I_j}x_i: \textrm{non-empty and disjoint }I_1,I_2,\dots,I_\ell\subset [k]\right\}\]  is monochromatic under $C$.
\end{thm}
\noindent

We would like to emphasize that our methods are \textit{heavily} based off of the work of Alweiss \cite{alweiss}. The contribution of this paper comes from several technical tricks and optimizations, which helped the powerful techniques from \cite{alweiss} attain (or at least approach) their ``true potential''. 

We feel that the techniques of \cite{alweiss} are very beautiful, and a secondary \hide{purpose}intention of this paper \hide{is}was to present them in a clearer and more polished fashion (so that they may be appreciated by a wider audience). Unfortunately, due to the highly technical nature of \cite{alweiss}, there is still noticeable room for improvement in this area. The patient reader might wish to revisit this preprint a bit later, once it has been updated.

\begin{ack}
    We are grateful to Ryan Alweiss for sharing a copy of his preprint (\cite{alweiss}) with us, before its appearance on the arXiv (and additionally for encouraging us to publish this independently). We additionally thank Alweiss for a number of conversations about a previous, related paper of his (\cite{alweiss2}) which had helped the author build intuition for some of the techniques which appear here.

    Lastly, we thank Ryan Alweiss, Daniel Carter, and Neil Hindman for taking a look at a preliminary version of this manuscript, and offering helpful feedback (not all of which have been implemented yet). Any shortcomings of the current manuscript land squarely on the author.
\end{ack}

\hide{\subsection{A brief comment on the structure}

The current note could perhaps be considered a research announcement. There are a number of places where we wish to
\begin{itemize}
    \item add clear motivation and/or high-level discussion;
    \item more smoothly present and organize several definition-heavy reductions, where we claim certain identities/embeddings behave as desired;
    \item fill-in the details of Theorem~\ref{rational main}, which is rather technical.    
\end{itemize}

}
\section{A totally different Ramsey problem}\label{family disjoint union}

\subsection{Some quick definitions}
Let $\vec{x} = (x_1,\dots,x_n)\in \mathbf{R}^n$ be a vector of semi-ring\footnote{For us, a (commutative) \textit{semi-ring} will be a set $\mathbf{R}$ equipped with two binary operations $+,\cdot:\mathbf{R}\times \mathbf{R}\to \mathbf{R}$ that satisfy commutativity, associativity and the distributive law. We do \textit{not} assume semi-rings have additive or multiplicative identities.} elements (say with $\mathbf{R}=\N,\Q_+$, or some more complicated setting). If $\II$ is a non-empty family of disjoint non-empty subsets $I\subset [n]$, we define
\[\varphi_\II(\vec{x}) := \sum_{I\in \II}\prod_{i\in I} x_i. \]For the remainder of this paper, we call such $\II$ \textit{$n$-families}.

These next definitions are not essential for the context of our paper, but it will aid discussion.

\begin{defn}
    Given an $n$-family $\II$ and $|\II|$-family $\JJ$, we define their composition
    \[\JJ\circ \II := \{\bigcup_{j\in J} I_j:J\in \JJ\}\]which is an $n$-family.
\end{defn}

\begin{defn}
    We say an $n$-family $\II$ is \textit{extreme} if $|\II| =1$ or $\max_{I\in \II}\{|I|\} = 1$ (i.e., either $\II$ only has one part, or all parts are singletons).
\end{defn}

\subsection{Discussion and reductions} \zh{TODO: fix `continuity errors' throughout this subsection. (namely, rephrase things to reflect that we actually can prove Theoremr~\ref{rational main} (meaning don't call it a conjecture, etc.))}

Given an $r$-coloring $C:\mathbf{R} \to [r]$, our goal (in order to re-prove Theorem~\ref{alw result}) is to find $\vec{x} =  (x_1,\dots, x_k)\in \mathbf{R}^k$ (with $\mathbf{R}= \Q_+$) such that the collection \begin{equation}\label{easier goal}
    \{\varphi_{\II}(\vec{x}):\II \textrm{ is an extreme $k$-family}\}
\end{equation}is monochromatic under $C$. Meanwhile, to prove Theorem~\ref{rational main}, one needs to find $\vec{x} \in \mathbf{R}^k$ where the collection \begin{equation}\label{harder goal}
    \{\varphi_{\II}(\vec{x}):\II \textrm{ is a $k$-family}\}
\end{equation}is monochromatic under $C$.

Now, given an $r$-coloring $C$ on $\mathbf{R}$ along with a vector $\vec{v} = (v_1,\dots,v_n) \in \mathbf{R}^n$, one obtains an $r$-coloring $C' = C_{\vec{v}}'$ on the collection of all $n$-families, via 
\[C'(\II) := C(\varphi_\II(\vec{v}))\]for each $n$-family $\II$. From here, to acheive the desired goal\zh{there are two goals! make less ambiguous}, it would suffice to find non-empty disjoint sets $I_1,\dots, I_k$, such that the collection 
\[\left\{\JJ\circ \{I_1,\dots,I_k\}: \JJ\textrm{ is an extreme $k$-family}\right\} \]
is monochromatic under $C'$ (indeed, then we would ``win'' by taking $x_j := \prod_{i\in I_j} v_i$ for $j=1,\dots,k$).

So now, we've reached a combinatorial coloring problem. If one could show that for every $r,k<\infty$, there existed some $n$ such that all $r$-colorings $C'$ of $n$-families, we had a monochromatic pattern like above, then our main theorem would follow immediately. However, such a result is impossible. Indeed, if $C'(\II)$ was determined by the parity of $|\II|$, then it would be impossible to find such a pattern for $k=2$ (thus it is already doomed in the case of $(r,k) =(2,2)$, in which case the collections appearing in (\ref{easier goal}) and (\ref{harder goal}) are the same). 

Nevertheless, we will soon see that if $C'(\II)$ exhibits some ``local structure'', then things work out much nicer.

Given a $k$-family $\II$, let\footnote{This is somewhat ad hoc, but perhaps not totally unexpected. In Ramsey theory it frequently is helpful to introduce a sense of order upon various subobjects that are considered.} $f(\II)$ be the set $I\in \II$ such that $\max(I)> \max(I')$ for all other $I' \in \II$. Let's first suppose we had an $r$-coloring of $n$-families, $C'$, such that
\[C' = c\circ f\]for some $r$-coloring $c$ of the non-empty subsets of $\{1,\dots,n\}$. In this case, we actually do win (for large enough $n$). This is due to a result known as ``the Disjoint Unions Theorem''.
\begin{thm}[Disjoint Unions Theorem]\label{DUT}
    For each $r,k<\infty$, there exists $n$ such that for any $r$-coloring $c$ of the subsets of $\{1,\dots,n\}$, there exists disjoint $I_1,\dots,I_k$ where
    \[\left\{\bigcup_{j\in J} I_j:\textrm{non-empty }J\subset [k]\right\}\]is monochromatic under $c$.
\end{thm}\noindent From here, it is not hard to show that
\[\{\JJ\circ \{I_1,\dots,I_k\}:\JJ\textrm{ is a $k$-family}\}\]is monochromatic under $C'$. \zh{TODO: clarify that we can reach this scenario. but this is a rather technical matter. so we consider the simpler Alweiss set-up primarily.}

A key insight of Alweiss is that a somewhat weaker assumption sufficed, namely that if $C'|_{\mathfrak{X}} = c\circ f|_{\mathfrak{X}}$ for some appropriately structured collection $\mathfrak{X}$ of $n$-families, then we still win. We shall prove a variant of this, for a slightly different structured collection $\mathfrak{X}'$ (this variant is not strictly necessary, but will make life slightly easier).

Here is what we use to re-prove Theorem~\ref{alw result}.
\begin{lem}\label{psuedo disjoint union}
    Fix $r,k<\infty$. Then there exists a finite integer $n$ such that the following holds:

    Consider any $r$-coloring $C'$ of the set of $n$-families. Also, suppose there is a coloring $c$ of the subsets of $\{1,\dots, n\}$, such that 
    \[C'(\II) = c\circ f(\II),\]for every $n$-family $\II$ where $|f(\II)| \le |I'|$ for all $I'\in \II$. 
    
    Then, there exists disjoint $I_1,\dots, I_k\subset [n]$ such that the desired pattern the collection
    \[\{\JJ\circ \{I_1,\dots,I_k\}:\JJ\textrm{ is an extreme $k$-family}\}\] is monochromatic under $C'$.
    \begin{rmk}
        Despite working with a different ``structured collection'', our proof will be essentially same as in \cite{alweiss}. We further note that, in turn, Alweiss's proof is very reminiscent of a proof of the ``Disjoint Unions Theorem'' presented by Graham, Rothschild, and Spencer \cite{graham}.
    \end{rmk}
\begin{proof}
        Fix some integer $n$, which we assume to be sufficiently large. Suppose we are given colorings $C',c$ as above. \zh{TODO: write quantifiers more precisely!!} 
        
        First, by repeatedly applying Ramsey's Theorem (and its hypergraph variants), we can find $S\subset \{1,2,\dots,n\}$ where:
        \begin{itemize}
            \item there exists an $r$-coloring $\chi:[|S|]\to [r]$ where $c(I)= \chi(|I|)$ for every subset $I \subset S$;
            \item $|S| \ge n'$, where $n'$ is some quantity which gets arbitrarily large as $n\to \infty$.
        \end{itemize}

        Next, applying \textit{Folkman's Theorem} (which is a generalization of Schur's Theorem), we can find a vector $\vec{m}\in \N^{n''}$ where:
        \begin{itemize}
            \item $n''$ is some quantity which gets arbitrarily large as $n'\to \infty$; 
            \item the set  \[\left\{\sum_{j\in J} m_j: \textrm{non-empty }J\subset [n''] \right\}\]is monochromatic under $\chi$ (i.e., contained inside a color class of $\chi$).
        \end{itemize}

        Whence, by taking $n$ sufficiently large, we can assume $n''\ge k$. So now, assume WLOG that that we $S\subset [n]$ and $\vec{m}\in \N^k$ satisfying the conditions above. We are nearly done. 

        By reordering our entries, we may assume $m_1\ge m_2\ge \dots \ge m_k$. Because of our last bullet, we must have $|S| \ge m_1+\dots+ m_k$ (otherwise $\chi$ would not be well-defined here). 
        
        Finally, writing $s_i$ to denote the $i$-th smallest element of $S$, we define 
        \[I_j:= \{s_i:\sum_{t=1}^{j-1} m_t < i \le \sum_{t=1}^j m_t\}\](this is to ensure that: \[\max(I_1)<\min(I_2)\le \max(I_2)<\min(I_3)\le \dots \le \max(I_{k-1})<\min(I_k).)\]
        
        We claim this suffices. Indeed, consider consider an extreme $k$-family $\JJ$. If $|\JJ| =1$, then $\JJ = \{J\}$ and \[C'(\JJ\circ \{I_1,\dots,I_k\}) = c(\bigcup_{j\in J}I_j) = \chi(\sum_{j\in J} m_j) = \chi(m_1).\]Otherwise, $\JJ$ is a set of singletons, in which case, letting $j^* := \max\{j:\{j\}\in \JJ\}$ , we see
        \[C'(\JJ\circ \{I_1,\dots,I_k\}) = c(I_{j^*}) = \chi(m_{j^*}) = \chi(m_1)\](here we have used the fact that $m_1\ge \dots \ge m_k$, so that $|f(\JJ\circ \{I_1,\dots,I_k\})| =|I_{j^*}|\le  |I'|$ for all $I'\in \JJ\circ \{I_1,\dots,I_k\}$, which justifies the first equality (since this means $\JJ$ is an $n$-family where $C'(\JJ) = c\circ f(\JJ)$, by our initial assumptions about $C',c$)).
    \end{proof}
\end{lem}

\section{A reduction and outline of the rest of the paper}

\subsection{More definitions}

The following definitions shall play a central role in the remainder of our paper.

\begin{defn}
    Let $C$ be some partial-coloring of some semi-ring $\mathbf{R}$. Also, let $\mathfrak{X}$ be some collection of $n$-families.
    
    We say a vector $\vec{v}\in \mathbf{R}^n$ is \textit{$\mathfrak{X}$-consistent} (with respect to $C$), if\footnote{Here and throughout the paper, we adopt the convention that statements of the form `$C(x) = C(y)$' are treated as false unless $C(x),C(y)$ are well-defined under the partial-coloring $C$.}
    \[C(\varphi_\II(\vec{v})) =C(\prod_{i\in f(\II)}v_i)\]for every $n$-family $\II\in \mathfrak{X}$. (meaning the color of $\varphi_\II(\vec{v})$ is controlled by the color of it's ``leading term'' (i.e., $\varphi_{\{f(\II)\}}(\vec{v})$)). 
\end{defn}

\begin{defn}
    Let $\mathfrak{X}_{lower}$ denote the collection of $n$-families $\II$ where $|f(\II)|\le |I'|$ for all $I'\in \II$. We shall say \textit{lower-consistent} as shorthand for $\mathfrak{X}_{lower}$-consistent. 
\end{defn}

The next definition is inspired by Ramsey-theoretic ``arrow notation''.  
\begin{defn}
    Given integers $n$, a collection of $n$-families $\mathfrak{X}$, and set of semi-ring elements $S$, we say
    \[S[\mathfrak{X}] \to (n;r)\]
    if: for every $r$-coloring $C:S\to [r]$, we can find a vector $\vec{v}\in S^n$ which is $\mathfrak{X}$-consistent with respect to $C$.
\end{defn}

\subsection{Plan of attack}

It is the goal for the remainder of this paper to prove two propositions (see previous subsection for definitions). They will imply our main results, as we shall discuss momentarily.
\begin{prp}\label{int build v} 
    For every $r,n<\infty$, there exists some finite $S\subset \Q_+$ such that
    \[S[\mathfrak{X}_{lower}]\to (n;r).\]
\end{prp}
\noindent With a bit more work, we can generalize our methods, and further prove that:
\begin{prp}\label{rat build v}
    For every $r,n<\infty$, there exists some finite $S\subset \Q_+$ such that
    \[S[\mathfrak{X}_{all}]\to (n;r),\]where $\mathfrak{X}_{all}$ is the collection of all $n$-families.
\end{prp}
\noindent In light of Section~\ref{family disjoint union} (specifically, Lemma~\ref{psuedo disjoint union} and Theorem~\ref{DUT}), we see that Proposition~\ref{int build v} will re-prove Theorem~\ref{alw result} and Proposition~\ref{rat build v} will imply Theorem~\ref{rational main}.

\hide{As one might guess, our proofs of the two propositions above are of a very similar flavor. We \zh{FINISH TRANSITION OR DELETE}}

Our strategy will use lots of auxiliary colorings and \hide{consequences of compactness\footnote{Note that in the statement of Lemma~\ref{build v}, we would be equally happy if $S$ was taken to $\mathbf{R}$, rather than a finite set. However, it will be crucial to our induction approach that $S$ can be taken to be finite. Of course, standard compactness arguments tell us that these are basically equivalent (unless, potentially $\mathbf{R}$ is uncountable, which is not very relevant to the current paper); however, we have written this assumption to emphasize its importance here.}}clever back-tracking during an inductive process, in a rather delicate manner. Indeed, at a glance (or when discussed informally), it is not obvious that our argument is non-circular. 

\hide{With this in mind, we will try to take care explaining things. In the remainder of subsection, we shall attempt to outline the key steps of the proof. Even then, the reader most likely will only have a rather fuzzy idea of what is going on.\zh{TODO: add a careful overview of things}}In a future draft, we will take care to motivate what is going on; but for the time being we apologize for presenting jarring and unmotivated proofs.

\hide{Rather than dive immediately into the nitty-gritty details of the more technical Theorem~\ref{integer main}, we shall first dedicate a subsection to the case of Theorem~\ref{rational main} an easier}

\section{How to induct}

\subsection{A technical result}
For convenience, we say an $n$-family $\II$ is called ``new'' if $n\in f(\II)$ (our reasoning for this phrase is that during induction, these correspond to the new constraints we need to worry about). Given a collection $\mathfrak{X}$ of $n$-families, we define $\newp(\mathfrak{X})$ to be set the of $(A,B)$ such that $\{A, B\cup \{n\}\}\subset \II$ for some $\II\in \mathfrak{X}$ (in some sense, fractions of the form `$\frac{\prod_{a\in A}v_a}{\prod_{b\in B}v_b}$' are generators of something which we will see is very important). 

Our key technical result is the following.
\begin{prp}[Stable extension]\label{stable ext}
    Let $\mathfrak{X}$ be a collection of $n$-families $\II$ which are all new. 
    
    Furthermore, suppose that $|A| > |B|$ for each $(A,B)\in \newp(\mathfrak{X})$.

    Then, for every $r<\infty$ and finite $Q\subset \Q_+$, there exist finite sets $D \subset \N,Q' \subset \Q_+$ and $\Lambda\subset \Q_{\ge 0}^{\newp(\mathfrak{X})}$, such that:

    Given any $\vec{u}\in \Q_+^{n-1}$, $x\in \Q_+$ and $r$-coloring $C:\Q_+\to [r]$, we can find $d\in D, q' \in Q'$ and $\vec{\lambda}\in \Lambda$ and define
    \[x':= q'\cdot \left(x+\sum_{(A,B)\in \newp(\mathfrak{X})}\lambda_{(A,B)}\frac{\prod_{a\in A}u_a}{\prod_{b\in B}u_b} \right),\]
    \[\vec{v}:= (d\cdot  u_1,\dots, d\cdot u_{n-1}, x'),\]so that $q\cdot \vec{v}$ is $\mathfrak{X}$-consistent (with respect to $C$) for every $q\in Q$. 
\end{prp}
\begin{rmk}
    Here, the phrase `stable' refers to how the $\mathfrak{X}$-consistency of $\vec{v}$ is `stable' under dilating by any $q\in Q$. \zh{TODO: adjust wording. I don't like how this sentence reads.}
\end{rmk}
\thide{\begin{rmk}\label{homogen discuss}
    Note that we cannot ask for more, in the sense that we cannot ensure that 
    \[C( \varphi_\II(q\cdot \vec{v})) = C(\varphi_{\{f(\II)\}}(\vec{v}))\]for all $\II\in \mathfrak{X},q\in Q$. Essentially this is because dilation is a ``homogenous transformation''; we can either enforce that \zh{explain more}
\end{rmk}}

\noindent We defer the proof of Proposition~\ref{stable ext} to the remainder of the paper. Instead, here we now show how to deduce our main theorems with some \hide{clever}routine induction.

\begin{rmk}\label{unnecessary}
    For the purposes of proving Hindman over $\Q$, it is only important that we can find some $d \in D$ and some $x'\in \Q_+$ so that $q\cdot \vec{v}$ is $\mathfrak{X}$-consistent for all $q\in Q$.

    We include the above statement (which says that we can start with any $x\in \Q_+$, and find a desired $x'$ as some mild tweak of $x$) for two reasons. Firstly, this additional property will help motivate the proof of Proposition~\ref{stable ext}, especially the details in Section~\ref{shifting lemmas}.

    Secondly, this additional property hints at potential applications for the integer setting. Indeed, if $u_1,\dots,u_{n-1},x$ are elements of $\N$, which satisfy certain divisibility conditions, then we can ensure that $x'\in \N$. The current issue with carrying this through, is that it is very difficult to get the new $x'$ to have any divisibility relations between other coordinates of $\vec{v}$ (which prevents us from carrying out the induction).
\end{rmk}

\hide{\begin{rmk}
    The key property of the above proposition (which allows us to extend Alweiss's result to the integer setting) is that the sets $Q^*,\Lambda$ depend only on $\mathfrak{X}$ and $r$ (and are independent of $x,\vec{u},$ and $d_0$). This means that we can predict what $\vec{v}$ will look like, up to finitely many simple ``perturbations''. This allows us to pick $x,\vec{u},d_0$ cleverly in advance, so that no matter the perturbation, we end up at a desirable outcome.
\end{rmk}

\begin{rmk}
    If one looks at Alweiss's ``shifting argument'', and verifies that it indeed succeeds (which is somewhat non-trivial, as it is a complicated process), then it is not too hard to convince oneself that the stated conclusions of Proposition~\ref{stable ext} must hold (by analyzing said shifting argument, due to its finitary nature). The tricky part, is perhaps considering these conclusions in the first place, and realizing they are useful (the inclusion of the factor $d_0$ is absolutely \textit{critical}, and very easy to overlook). 
\end{rmk}}

\subsection{Establishing our main theorem}\label{the induction}\hide{We will use Proposition~\ref{stable ext} to stay inside the integers, by taking $d_0$ to absolutely huge, causing each of the `steps' (coming from summands in the definition of $x'$) to lie inside some large lattice. 

More formally, we shall now establish the following result, which clearly implies Proposition~\ref{int build v}, and thus establishes our main result (Theorem~\ref{integer main}).}

\begin{prp}\label{funny induction}
    Fix any choice of (positive) integers $r,n<\infty$, along with a finite set $Q\subset \Q_+$.

    For every $r$-coloring $C:\Q_+\to [r]$, we can find $\vec{v}\in \Q_+^n$ such that $q\cdot \vec{v}$ is lower-consistent with respect to $C$ for all $q\in Q$.

    \thide{\begin{rmk}
       The below proof is a little longer than it needs to be. This is because it came from editing a faulty proof of Hindman over $\N$ (where I had some additional requirements in the induction hypothesis).  \zh{TODO: simplify this proof, it should honestly be 3 paragraphs or something. afterwards, delete this Remark!}
    \end{rmk}}
    \begin{proof}
        We write $\CC_n$ to denote the claim that: for every finite $Q\subset \Q_+$, and finite-coloring $C:\Q_+\to [r]$, that we can find $\vec{v}\in \Q_+^n$ such that $q\cdot \vec{v}$ is consistent with respect to $C$ for all $q\in Q$.

        We have that $\CC_1$ is true, by just taking any 1-dimensional vector $\vec{v} \in \Q_+^{1}$ (since vectors of length one are vacuously consistent with respect to all colorings $C$).

        Now, assume $\CC_n$ for some $n\ge 1$. We will deduce $\CC_{n+1}$. Upon establishing thing, Proposition~\ref{funny induction} clearly follows by induction (completing our proof). 
        
        Let $\mathfrak{X}= \mathfrak{X}_{n+1}$ be the collection of $(n+1)$-families $\II$ where there does not exist $I'\in \II$ with $I'\neq f(\II)$ and $|f(\II)|< |I'|$. Write
        \[\mathfrak{X} = \mathfrak{X}_{\text{new}}\sqcup \mathfrak{X}_{\text{old}},\]where $\mathfrak{X}_{\text{new}}$ is the subcollection of new $(n+1)$-families $\II\in \mathfrak{X}$ (meaning that $(n+1)\in f(\II)$), and $\mathfrak{X}_{\text{old}}$ is the remaining families in $\mathfrak{X}$.
        
        Write $\Omega := \newp(\mathfrak{X}_{\text{new}})$, and note that by definition, we will have $|A|<|B|$ for each $(A,B)\in \Omega$. Indeed, for each such pair $(A,B)$, there exists $\II\in \mathfrak{X}_{\text{new}}$ where $\{A,B\cup \{n+1\}\}\subset \II$, and we'll have that
        \[|B| = f(\II)-1< |A|\](since $A=I'$ for a set $I'\in \II$ with $I'\neq f(\II)$). 

        Now consider any pair $\pi = (r,Q)$, where $r<\infty$ and $Q\subset \Q_+$ is a finite subset. 
        
        By Proposition~\ref{stable ext}, we can find\footnote{Technically $Q_*$ will also depend on $n$, but we are supressing this dependence, as $n$ is that to be fixed here.} a finite set \[Q_* = Q_*(\pi)\subset\N\subset  \Q_+,\] such that, given any $\vec{u}\in \Q_+^n$ and $C:\Q_+\to [r]$, we can pick any $q_* = q_*(\vec{u},C) \in Q_*$ and $x'(\vec{u},C) \in \Q_+$ where defining:
        \[\vec{v} := (q_*\cdot u_1,\dots, q_*\cdot u_n,x')\]
        we get that $q\cdot \vec{v}$ is $\mathfrak{X}_{\text{new}}$-consistent for all $q\in Q$.

        We must now argue (considering $\pi=(r,Q)$ fixed), that for every $r$-coloring $C:\Q_+\to [r]$, we can find $\vec{v}\in \Q_+^{n+1}$ where $q\cdot \vec{v}$ is $\mathfrak{X}$-consistent (with respect to $C$) for all $q\in Q$.

        First take $Q_* = Q_*(\pi)$, which is a finite subset of $\Q_+$. Also let $Q' := \{qq_*:q\in Q,q_*\in Q_*\}$, which is clearly a finite set. 
        
        Next, fix a coloring $C:\Q_+\to [r]$.
        
        By $\CC_n$, we know that there must exist some $\vec{u}\in \Q^n$ where $q'\cdot \vec{u}$ is $\mathfrak{X}_{\text{old}}$-consistent (with respect to $C$) for all $q'\in Q'$.

        Next, consider $q_*= q_*(\vec{u},C),x' = x'(\vec{u},C)$, and define 
        \[\vec{v} = (q_*\cdot u_1,\dots, q_*\cdot u_n,x')\in \Q_+^{n+1}.\]We claim that $q\cdot \vec{v}$ is $\mathfrak{X}$-consistent for each $q\in Q$. Suppose not, then there would exist some $\II\in \mathfrak{X}$ and $q\in Q$ where
        \[C(\varphi_\II(q\cdot \vec{v})) \neq C(\varphi_{\{f(\II)\}}(q\cdot \vec{v})).\]By construction of $\vec{v}$, the above situation is impossible when $\II \in \mathfrak{X}_{\text{new}}$; so we suppose $\II\in \mathfrak{X}_{\text{old}}$. Since $\vec{v}|_{[n]} = q_*\cdot \vec{u}$, this would mean that
        \[C(\varphi_{\II}((qq_*)\cdot \vec{u})) \neq C(\varphi_{\{f(\II)\}}((qq_*)\cdot \vec{u}),\]which implies $(qq_*)\cdot \vec{u}$ is not $\mathfrak{X}_{\text{old}}$-consistent (with respect to $C$). But this is also a contradiction, due to our choice of $\vec{u}$, since $qq_*\in Q'$. 
        
        It follows that $\vec{v}$ does indeed have the desired properties, concerning $Q$. And so, since $r,Q$ can be arbitrarily chosen, we see $\CC_{n+1}$ holds, completing our proof.
    \end{proof}
\end{prp}

\section{Overview}

\subsection{A digression into yet another Ramsey problem}\label{digression}

We will start from a very general perspective. We note that, despite the generality of this subsection, and we will still end up needing to work with a modified version of the below (covered in Section~\ref{finer multitask}), thus the following should be viewed as a `toy model' for what we will do later on.

Let $X$ be some ground set, and let $\PP$ be a set of maps  $p:X\to X$, which is closed under composition (i.e., $p_1,p_2\in \PP\implies p_1\circ p_2\in \PP$). We will think of the elements $p\in \PP$ as being `perturbations' of our ground set. 

Also, we will often write $\CC$ to denote some collection of colorings of $X$. Given integer $r$, we write $\CC_{\le r}$ to denote the set $\{C\in \CC: C(X)\subset [r]\}$.

Now, let's introduce two Ramsey properties.
\begin{defn}\label{compact definition}
    Given a collection $\CC$ of colorings of $X$ and sets $P,P'\subset \PP$, we say $P'\to_\CC P$, if for every $x\in X$ and $C\in \CC$, we can find $p'\in P'$ such that the set
    \[\{p\circ p'(x):p\in P\}\]is monochromatic under $C$.

    We say that a triple $(X,\PP,\CC)$ is \textit{compactly Ramsey}, if for every $r<\infty$ and finite $P\subset \PP$, there exists a finite $P'
    \subset \PP$ where
    \[P'\to_{\CC_{\le r}} P.\] 
\end{defn}

\begin{rmk}\label{allusion}
    It turns out that a triple $(X,\PP,\CC)$ being compactly Ramsey is actually a very restrictive property (cf. Remark~\ref{too strong}). In lieu of this, it is hard to give examples, without placing complex stipulations on $\CC$. 
    
    However, if we take $X := \N$ and $\CC$ to be the set of all colorings of $\N$, then $(X,\PP,\CC)$ will be compactly Ramsey if $\PP$ is the collection of maps
    \[p_m :n\mapsto \max\{n,m\}\]for $m\in \N$.
\end{rmk}

\begin{defn}
    We say a triple $(X,\PP,\CC)$ is \textit{multi-task Ramsey}, if, for each $r<\infty$ and finite $P\subset \PP$, there exists a finite $P'\subset \PP$ where: 

    For every choice of $x\in X$ and colorings $(C_p\in \CC_{\le r})_{p\in P}$, we can find $p'\in P'$ such that
    \[C_p(p\circ p'(x))= C_p(p'(x))\]for all $p\in P$. (when this happens, we say that $P' \implies P^{\CC_{\le r}}$)
\end{defn}

We observe that the former condition implies the latter. The secret, it to use a lot of back-tracking.
\begin{prp}\label{compactness implies multitask}
    Suppose a triple $(X,\PP,\CC)$ is compactly Ramsey. Then it is also multi-task Ramsey.
\end{prp}
The proof will appear a bit later in the subsubsection. Before this, we will give a bit more discussion of some helpful concepts. 

Given $x\in X$ and $P\subset \PP$, one may find it useful to interpret the set
\[\{p(x): p\in P\}\]as being ``the $P$-neighborhood around $x$''.
Then the statment: 
\[\textrm{``we have
$C(p(x)) = C(x)$
for all $p\in P$,''}\]
\noindent now says (possibly after adding the identity map to $P$, which doesn't really matter): \[\textrm{``there exists some color class $\chi\subset X$ of $C$, so}\]\[ \textrm{that the $P$-neighborhood around $x$ is inside $\chi$.''}\]

    \begin{proof}[Proof of Proposition~\ref{compactness implies multitask}]
        We proceed by induction on inclusion.

        When $P=\emptyset$, there is nothing to prove (we can take $P' = \emptyset$ independently of $r$).

        Now suppose that for some $P\subset \PP$, and every $r<\infty$, that there exists a finite $P'= P'(r)$ such that $P' \implies P^{\CC_{\le r}}$. We claim that for each $p^*\in \PP$ and $r<\infty$, there exists $P''= P''(r)$ such that $P''\implies (P \cup \{p^*\})^{\CC_{\le r}}$. \dc{It might be nice to expand out this last definition to match what you write in the conclusion of the inductive step, i.e. ``We wish to find $P''$ so that for all $x,(C_q)_{q\in P\cup\{p^*\}}$, ...'' or whatever}

        Indeed, fix $r<\infty$ and set $P' := P'(r)$. With foresight, let $P_{new}:= P' \cup \{p^*\circ p':p'\in P'\}$. As our triple is compactly Ramsey, we can find a finite $P_{new}'$ such that $P_{new}'\to_{\CC_{\le r}} P_{new}$. Now, we claim that taking $P'' := \{p' \circ p_{new}':p'\in P',p_{new}'\in P_{new}'\}$ (which is clearly finite) suffices. 

        Indeed, for each $x\in X$ and possible outcome of $C_{p^*}\in \CC_{\le r}$, we can pick $p_{new}'\in P_{new}'$ so that writing $x' := p_{new}'(x)$, we have
        \[C_{p^*}(p^*\circ p'(x') )= C_{p^*}(p'(x'))\]for each $p'\in P'$ (since $p^*\circ p'$ and $p'$ both belong to $P_{new}$ for each $p'\in P'$).

        Meanwhile, by definition of $P'$, for each possible outcome of $(C_p\in \CC_{\le r})_{p\in P}$, we can pick $p'\in P'$ so that
        \[C_p(p\circ p'(x')) = C_p(p'(x'))\]for each $p'\in P'$. By design, writing $x'' := p'(x')$  \dc{so $x''=p''(x)$ for some $p''\in P''$}, we get that for each $q \in P \cup \{p^*\}$, that $C_q(q\circ x'') = C_q(x'')$ as desired \dc{Better to write $C_q(q\circ p''(x)) = C_q(p''(x))$}.
        
        Moreover, we have that $p'\circ p_{new}' \in P''$ by definitions.
    \end{proof}

\begin{rmk}
    The key `point' was that we could make $P_{new}\subset \PP$ be as large as we please. Thus, since there are only finitely many perturbations $p'$ to worry about, when we handle the colorings for $p\in P$, we can just ``look into the future'' and add elements to $P_{new}$ so that we still win after each possible outcome of $p'$. \thide{This is perhaps nicer to interpret from the neighborhood perspective, \zh{TODO: phrase this explanation better!}}

    Not only is the idea implicit in the recent work of Alweiss \cite{alweiss}, but it also appears (again implicitly) in an earlier paper of Alweiss (namely \cite{alweiss2}).
\end{rmk}

\zh{TODO: draw parralels to the use of being stable under dilations in the previous section, and being stable by perturbations in this section. basically this is just a way more general framework for the last proof.}

\subsection{Motivation}\label{motivation}
Let $\vec{u} = (u_1,\dots,u_{n-1})\in \Q_+^{n-1}$ be a vector. Also, consider some starting point $x\in \Q_+$.

Now consider some new $n$-family $\II = \{A_1,\dots,A_\ell,S\cup \{n\}\}$. We define the `$\II$-shift' operator, as
\[\sigma_\II:\Q_+^{n-1}\times \Q_+ \to \Q_+^{n-1}\times \Q_+;(\vec{u},x)\mapsto (\vec{u},x+\sum_{i=1}^\ell \frac{\prod_{a\in A_i}u_a}{\prod_{s\in S} u_s}).\]We note the identity\footnote{Here as an abuse of notation, we identify pairs $((u_1,\dots,u_{n-1}),x)\in \Q_+^{n-1}\times \Q_+$ with their `concatenation' $(u_1,\dots,u_{n-1},x)$.}\[\varphi_\II((\vec{u},x)) = \varphi_{\{f(\II)\}}(\sigma_\II((\vec{u},x)));\]while this might seem strange, the `point' of this definition is that we can now express notions of consistency in terms of a Ramsey problem which is in the flavor of the last subsection (where we treat $\sigma_\II$ as a `perturbation', and consider colorings of $\Q_+^{n-1}\times \Q_+$ that are induced by evaluating $\varphi_{\{\II\}}$).

Specifically, suppose we have fixed some finite coloring $C$ of $\Q_+$.

For $S \subset [n-1]$, we define 
\[C_S((\vec{u};x)) := C(x\prod_{s\in S}u_s). \]Given a $n$-family $\II = \{A_1,\dots,A_\ell, S\cup \{n\}\}$, we have that $(u_1,\dots,u_{n-1},x)$ is $\{\II\}$-consistent (wrt $C$) if:
\[C_S((\vec{u};x)) = C_S(\sigma_\II((\vec{u};x_\II))).\]

Thus, Propostion~\ref{stable ext} would certainly be implied by the following `delusion' (i.e., false claim). 
\begin{delusion}\label{dream statement}
    Let $\mathfrak{X}$ be a set of new $n$-families. Write $\Omega := \newp(\mathfrak{X})$, and suppose $|A|>|B|$ for each $(A,B)\in \Omega$.

    Then, for every $r<\infty$ and finite $Q\subset \Q_+$, there exists a finite set $Q_*\subset \Q_+$ such that:

    Given any $\vec{u}$, along with $r$-colorings $C_\II:\Q_+^{n-1}\times \Q_+ \to [r]$ for each $\II \in \mathfrak{X}$, then there exists $q_*\in Q_*$ and $x\in \Q_+$, so that
    \[C_\II( \sigma_\II(q \cdot (q_*\cdot \vec{u},x))) = C_\II(q\cdot (q_*\cdot \vec{u},x))\]for each $\II\in \mathfrak{X}$.
\end{delusion}
\begin{rmk}
    The above statement is blatantly false. Indeed, when $\sigma_\II$ has no fixed points (i.e., when $|\II|>1$), we can simply define a coloring $C_\II:\Q_+^{n-1} \times \Q_+ \to [2]$ so that $C_\II(\sigma_\II((\vec{u},x)))\neq C_\II((\vec{u},x))$ for all $\vec{u},x$ (this is similar to the discussion in Remark~\ref{trivial N}, where an opponent can ensure that $\{q,2q\}$ is never monochromatic).

    To get around this, we establish a slightly weaker result (cf. Proposition~\ref{culmination}), which still suffices to `get the job done'. 
\end{rmk}

Now, the statement of the above fake result is very reminiscant to the notion of being \textit{multi-task Ramsey}, which we introduced in the last subsection. So, to establish Proposition~\ref{stable ext}, we will define a closed set of perturbations $\PP$ over $X:=\Q_{n-1}^+\times \Q_n$, which contains the $\II$-shift $\sigma_\II$ for each $\II\in \mathfrak{X}$ (and furthermore, also constains the $q$-dilation $(\vec{u},x)\mapsto (q\cdot \vec{u},qx)$ for each $q\in \Q_+$). Then, for an appropriate choice of colorings $\CC$ (which corresponds to ``inheritting'' some nonsense in the appropriate way \zh{TODO: word correctly!}), we will establish the triple $(X,\PP,\CC)$ is compactly Ramsey, which will imply Proposition~\ref{stable ext}.

However, the exact details of defining $\PP$ and $\CC$ are a bunch of technical mumbo-jumbo, and the author does not wish to bombard the reader with those complexities at this time. Instead, in the next section, we will prove a handy lemma (Lemma~\ref{general term shift}). Then finally, in the last section, we will formally define our system $(X,\PP,\CC)$, and reveal how Lemma~\ref{general term shift} easily implies this system has the desired `compactly Ramsey'-property, and then quickly deduce Proposition~\ref{stable ext}.   

\section{Handling additive shifts}\label{shifting lemmas}
\subsection{First lemma}

To prove Proposition~\ref{stable ext}, we will bootstrap the following ``shifting lemma''.
\begin{lem}\label{single shift}
    Let $\mathfrak{X}$ be a collection of $n$-families $\II$ which are all new. 
    
    Furthermore, suppose that $|A| > |B|$ for every $(A,B)\in \newp(\mathfrak{X})$.

    Then, for every fixed choice of $r<\infty$ and finite $\Xi \subset \Q_{\ge 0}^{\newp(\mathfrak{X})}$, there exist finite sets $D \subset \N$ and $\Delta \subset \N^{\newp(\mathfrak{X})}$, such that:

    Given any $\vec{u}\in \Q_+^{n-1}$, $x\in \Q_+$, and $r$-coloring $C:\Q_+\to [r]$, we can find $d\in D$ and $\vec{\delta}\in \Delta$ and define
    \[x':= x+\sum_{(A,B)\in \newp(\mathfrak{X})}\delta_{(A,B)}\frac{\prod_{a\in A}u_a}{\prod_{b\in B}u_b} ,\]
    \[\vec{u}':= d \cdot \vec{u} = ( du_1,\dots, du_{n-1}),\]so that:
    \[C(x'+ \sum_{(A,B)\in\newp(\mathfrak{X})}\xi_{(A,B)}\frac{\prod_{a\in A} u_a'}{\prod_{b\in B}u_b'}) = C(x')\]for each $\vec{\xi}\in \Xi$. 
\end{lem}
\hide{\begin{rmk}
    It is very important here that the sets $D,\Delta$ are finite, and depend only on the choice $r,\Xi,d_0$.
\end{rmk}
\begin{rmk}
    The author thinks it is fair to say that the most important idea of Alweiss's paper, is that some qualitative form of Lemma~\ref{single shift} holds in a slightly different setting (as mentioned before Lemma~\ref{psuedo disjoint union}, Alweiss worked a different collection $\mathfrak{X}$ of $n$-families, which leads to different assumptions concerning the pairs $(A,B)$). However, in \cite{alweiss}, the exact finitary nature of $D,\Delta$, and their precise dependence on various parameters is essentially glosssed over. And more importantly, in \cite{alweiss}, they did not realize the powerful fact that we can introduce this scaling factor $d_0$.
\end{rmk}}

To prove this lemma, we require a powerful generalization of the van der Waerden numbers. First we must define some terminology concerning polynomials.
\begin{defn}
    We say a polynomial $p \in \Q_+[X]$ is \textit{good} if it has zero constant-term (i.e., the formal variable $X$ is a divisor of $p$). 
\end{defn}
\begin{defn}
    Let $\vec{p} = (p_i)_{i\in I} \in (\Q_+[X])^I$ be a tuple of polynomials. Given $d\in \Q_+$, we define the evaluation of $\vec{p}$ at $d$ as 
    \[\vec{p}(d) := \sum_{i\in I} p_i(d)e_i,\]where the terms `$e_i$' are considered to form the canonical basis of the lattice $\Z^I$ (here, our convention that $\N$ does not contain zero becomes a minor nuisance, since these elements $e_i$ do not belong to $\N^I$ (nor $\Q_+^I$), despite us wishing to work there; we shall somewhat gloss over this and apologize for the abuse of notation). 
\end{defn}

We can now state a celebrated result of Bergelson and Leibman \cite{bergelson}, which was later proved combinatorially by Walters \cite{walters}. 

\begin{thm}[Multidimensional polynomial van der Waerden] {\cite[Theorem~B]{bergelson}}
    Let $P \subset (\Q_{\ge 0}[X])^I$ be a finite set of multi-dimensional polynomials (with $I$ being a finite set). Now assume that for each $\vec{p}: = (p_i)_{i\in I}\in P$, we have:
    \begin{itemize}
        \item $p_i(d)\in \N$ for all $d\in \N$;
        \item and $p_i$ is good;
    \end{itemize}for each $i\in I$.

    Then, for every $r<\infty$, there exists some finite $S\subset \N^I$, so that, for all $r$-colorings $C:S\to [r]$, we can find $x\in S$ and $d\in \N$ such that
    \[C(x+\vec{p}(d)) = C(x)\]for all $\vec{p}\in P$.
\end{thm}
By what is essentially a change of variables, rewriting $X$ as $NX'$ for some $N\in \N$ with enough factors, one can easily drop the first condition about our set of polynomials. Then, by an additional rescaling, we can furthermore ask that our set $S$ lies within a rescaled lattice. We summarize this as follows.
\begin{thm}[Multidimensional polynomial van der Waerden (revised)]\label{mpvdw}
    Let $P \subset (\Q_{\ge 0}[X])^I$ be a finite set of multi-dimensional polynomials (with $I$ being a finite set). Now assume that for each $\vec{p} = (p_i)_{i\in I}\in P$, that we have $p_i$ is good for each $i\in I$.

    Then, for every $r,d_0<\infty$, there exists some finite\footnote{Here, $d_0\cdot \N\subset \N$ is the set of positive integers divisible by $d_0$.} $S\subset (d_0 \cdot \N)^I$, so that, for all $r$-colorings $C:S\to [r]$, we can find $x\in S$ and $d\in \N$ such that
    \[C(x+\vec{p}(d)) = C(x)\]for all $\vec{p}\in P$.
\thide{\begin{rmk}
    I am actually no longer using this $d_0$ assumption. It is fine to just assume $d_0 = 1$ always. (this will be rewritten in a future version)
\end{rmk}}
    \begin{proof}
        Suppose this was false, for some choice of $P,d_0$. By finiteness, there must exist some $d'\in \N$, such that for every $\vec{p}\in P$, there exists a corresponding $\vec{p}'\in (\Z_{\ge 0}[X])^I$, so that 
        \[\vec{p}(d' d) = d_0 \cdot \vec{p}(d)\] for all $d\in \N$. 
        
        By the previous theorem, fixing $r<\infty$ and $P' := \{\vec{p}':\vec{p}\in P\}$ (whose polynomials are always integer valued), there exists a finite $S'\subset \N^I$ where we can find a monochromatic pattern inside each $r$-coloring of $S'$. We then take $S = \{d\cdot x:x\in S'\}$.
    \end{proof}
\end{thm}

Since we will need a more general version of Lemma~\ref{single shift}, we will just give a ``high-level'' proof for now. This skips over a cumbersome calculation, but otherwise covers all the main ideas.
\begin{proof}[(Sketch of) Proof of Lemma~\ref{single shift}] Suppose we have fixed a collection $\mathfrak{X}$ where $|A|>|B|$ for all $(A,B)\in \newp(\mathfrak{X})$.

    We set $\Omega:= \newp(\mathfrak{X})$.

    Now given $\vec{\xi} \in \Q_{\ge 0}^{\Omega}$, we define the multi-dimensional polynomial  \[\vec{p}_{\vec{\xi}}: X\mapsto (\xi_{(A,B)} X^{|A|-|B|})_{(A,B)\in \Omega} \in (\Q_+[X])^{\Omega};\]recalling the assumption $|A|>|B|$, we see that each coordinate in this tuple is good.
    
    By Theorem~\ref{mpvdw}, applied with $r=r$ and $P = \{p_{\vec{\xi}}: \vec{\xi}\in \Xi\}$, there exists some finite $\Tilde{S} \subset \N^{\Omega}$ so that, for every $r$-coloring $\Tilde{C}:\Tilde{S}\to [r]$, we can find $\Tilde{x}\in \Tilde{S}$ and $d\in \N$ where
    \[C(\Tilde{x}+\vec{p}(d)) = C(\Tilde{x})\]for each $\vec{p}\in P$. Now, take $D$ to be the set of $d\in \N$ such that $\{\Tilde{x}+\vec{p}(d):\vec{p}\in P\}\subset \Tilde{S}$ for some  choice of $\Tilde{x}\in \Tilde{S}$. Clearly, $D$ is finite (since $\N\subset \Z$ has characteristic zero).
    
    We now suppose we are given $\vec{u}\in \Q_+^{n-1}$, $x\in \Q_+$, along with an $r$-coloring $C:\Q_+\to [r]$. 
        
    For $(A,B)\in \Omega$, we write $\rho_{(A,B)} := \frac{\prod_{a\in A} u_a}{\prod_{b\in B}u_b}$, to denote the associated ratio.
        
    We write $\vec{w} := X\cdot \vec{u}= (u_1 X,\dots,u_{n-1}X)$ (here $X$ is treated as a formal variable). For each pair of subsets $(A,B)\in \newp(\mathfrak{X})$, we have that
    \[\frac{\prod_{a\in A}w_a}{\prod_{b\in B} w_b} = \rho_{(A,B)} X^{|A|-|B|}, \]with the coefficient $\rho_{(A,B)}\in \Q_+$.

    Now we consider the homomorphism $\phi:\N^{\Omega}\to \Q_+$ induced by sending $e_{(A,B)}\mapsto \rho_{(A,B)} $. Our coloring $C$ over $\Q_+$, restricted to $\phi(\Tilde{S})$, yields a pull-back coloring over the set $\Tilde{S}$:
    \[\Tilde{C}:\Tilde{S}\to [r]; \Tilde{x}\mapsto C(x+\phi(\Tilde{x})).\]
    By design, the maps $\phi,\Tilde{C},C$, all relate to one in another in a useful way. This is shown rigorously in Claim~\ref{homorph calculation} (which we reach in few subsections), but for now, we will not get into specifics.

    By construction of $\Tilde{S}$, there will be some choice of $\Tilde{x}\in \Tilde{S}$ and $d\in D$ where $\{\Tilde{x}\} \cup \{\Tilde{x}+\vec{p}(d):\vec{p}\in P\}$ is monochromatic under $\Tilde{C}$. We claim that by taking
    \[x' = x+ \sum_{(A,B)\in \Omega}\phi(e_{(A,B)})\Tilde{x}_{(A,B)} = x+  \sum_{(A,B)\in \Omega}\Tilde{x}_{(A,B)}\rho_{(A,B)}\]
    \[\vec{u} := d \cdot \vec{u},\]then we will have
    \[C(x' + \sum_{(A,B)\in \Omega} \xi_{(A,B)}\frac{\prod_{a\in A}u_a'}{\prod_{b\in B} u_b'}) = \Tilde{C}(\Tilde{x}+\vec{p}_{\vec{\xi}}(d)) = \Tilde{C}(\Tilde{x})= C(x')\]for all $\vec{\xi}\in \Xi$ (this can be confirmed by using Claim~\ref{homorph calculation}, or perhaps the motivated reader will find this enlightening to work out themselves).
         
    If one has faith that the above claim is true, then they will see the conclusion of Lemma~\ref{single shift} is true, as we can take $D=D$ and $\Delta = \Tilde{S} \subset \N^\Omega$.
\end{proof}

\subsection{Stonger stability}

We introduce some handy notation. 
\begin{defn}
    Suppose we are working with some collection of new $n$-families, $\mathfrak{X}$, with $\Omega:= \newp(\mathfrak{X})$. Then, for $\vec{u}\in \Q_+^{n-1}$, we define the vector of ratios
    \[\vec{\rho}_{\vec{u}} := (\frac{\prod_{a\in A}u_a}{\prod_{b\in B} u_b})_{(A,B)\in \Omega}\in \Q_+^\Omega.\]
\end{defn}
\noindent The ``meaning'' behind the above definition, is that when we apply Lemma~\ref{single shift}, \[x'-x = \vec{\delta}\cdot \vec{\rho}_{\vec{u}}.\]

\subsubsection{Handling additive perturbations}
We now establish the following generalization of Lemma~\ref{single shift} (basically, the first lemma was the `$S = \emptyset$ case').

\begin{lem}\label{general term shift}
    Fix $S\subset [n-1]$.

    Let $\Omega$ be the set of pairs $(A,B)$ with $A,B\subset [n-1]$ with $|A|+|S|>|B|$.

    Then, for every choice of $r<\infty$ and finite set $\Xi \subset \Q_{\ge 0}^\Omega$, there exists finite sets $D\subset \N,\Delta\subset \N^\Omega$, such that:

    Given any $\vec{u}\in \Q_+^{n-1}$, $x\in \Q_+$, and $r$-coloring $C:\Q_+\to [r]$, we can find $d \in D,\vec{\delta}\in \Delta$ and define
    \[x' := \frac{x+ \vec{\delta}\cdot \vec{\rho}_{\vec{u}}}{d^{|S|}}\]
    \[\vec{u}' :=d\cdot \vec{u} = (du_1,\dots,du_{n-1}),\]
    \[\vec{v}:= (u_1',\dots,u_{n-1}',x'),\]
    so that for each $\vec{\xi}\in \Xi$, we have
    \[C(\prod_{s\in S} v_s \left( x'+ \vec{\xi}\cdot \rho_{\vec{u}'}\right)) = C(v_n\prod_{s\in S} v_s).\] 
    
\end{lem}

\begin{proof}
        Essentially, we just apply multi-dimensional polynomial van der Warden in a very similar way as in Lemma~\ref{single shift}.

        Given $\vec{\xi}\in \Q_{\ge 0}^\Omega$, we associate the multi-dimensional polynomial:
        \[\vec{p}_{\vec{\xi}}: X\mapsto (\xi_{(A,B)} X^{|A|-|B|+|S|})_{(A,B)\in \Omega} \in (\Q_+[X])^{\Omega};\]each coordinate of said polynomial is good (has constant term zero).

        So, let $P = \{\vec{p}_{\vec{\xi}}: \vec{\xi}\in \Xi\}$. Then for any $r<\infty$, we can invoke Theorem~\ref{mpvdw} to find a finite $\Tilde{S} \subset \N^{\Omega}$, where for every $r$-coloring $\Tilde{C}$ of $\Tilde{S}$, there exists $\Tilde{x}\in \Tilde{S}$ and $d\in \N$ where the set 
        \begin{equation}\label{pattern}
            \{\Tilde{x}\} \cup \{\Tilde{x}+\vec{p}(d):\vec{p}\in P\}
        \end{equation}is monochromatic.

        So, fix a choice of $r<\infty$, and a corresponding (finite) set $\Tilde{S}\subset \N^\Omega$.

        We then consider any $r$-coloring $C:\Q_+\to [r]$, and choice of $x\in \Q_+, \vec{u}\in \Q_+^{n-1}$. Define the homomorphism
        \[\phi: \Z^\Omega \to \Q_+;e_{(A,B)} \mapsto \prod_{s\in S} u_{s}\frac{\prod_{a\in A}u_a}{\prod_{b\in B}u_b}.\] We note the following identity holds.
        \begin{clm}\label{homorph calculation}
            Consider any choice of $x\in \Q_+,d\in \N$.

            We get that
            \[ \prod_{s\in S} d\cdot u_s \left(\frac{x+ \vec{\delta}_{\vec{u}}}{d^{|S|}} + \vec{\xi} \cdot \vec{\rho}_{d\cdot \vec{u}}\right) = x\prod_{s\in S} u_s  + \phi(\vec{\delta}) + \phi(\vec{p}_{\vec{\xi}}(d))\]
            for every $\vec{\delta}\in \N^\Omega,\vec{\xi}\in \Q_{\ge 0}^\Omega$.
            
            \begin{proof}
                We defer the details to a later draft.\hide{We will break each side of the claimed equality into sums on some common index set $T$, and then show that the individual terms in these sums ``cancel out''.} 

               \zh{TODO: give details (though it is a basic identity)}
            \end{proof}
        \end{clm}

        To complete the proof, we shall take $\Delta := \Tilde{S}$ and $D$ to be the (clearly finite) set of $d\in \N$ where there exists $\Tilde{x}\in \Tilde{S}$ such that $\{\Tilde{x}+ \vec{p}(d):\vec{p}\in P\}\subset \Tilde{S}$.
        
        Indeed, consider any $r$-coloring $C:\Q_+\to [r]$. This yields a pull-back coloring
        \[\Tilde{C}: \Tilde{S}\to [r]; \Tilde{x}\mapsto C(x\prod_{s\in S} u_s+\phi(\Tilde{x})).\]
        By construction of $\Tilde{S}$, there will be some choice of $\Tilde{x}\in \Tilde{S},d\in \N$, so that the pattern described in \eqref{pattern} is monochromatic (under $\Tilde{C}$).

        Consequently, taking
        \[x' := \frac{x+\Tilde{x}\cdot \vec{\rho}_{\vec{u}}}{d^{|S|}},\]
        \[\vec{u}':= d\cdot \vec{u},\]then the above claim says that
        \begin{align*}
            C( \prod_{s\in S}u_s' \left(x' + \vec{\xi} \cdot \vec{\rho}_{\vec{u}'}\right)) &= C\left(x\prod_{s\in S} u_s + \phi(\Tilde{x}) + \phi(\vec{p}_{\vec{\xi}}(d))\right) \\
            &= \Tilde{C}(\Tilde{x}+ \vec{p}_{\vec{\xi}}(d))\\
            &= \Tilde{C}(\Tilde{x}) = C(x'\prod_{s\in S}u_s').\\
        \end{align*}
    This establishes what we desired.\end{proof}

\section{The mumbo-jumbo}

\subsection{Some perturbations on the rationals}

\hide{In this subsection, we define some maps on $\Q_+^n$, and show they are closed under composition (so that we can apply Proposition~\ref{compactness implies multitask} in the next subsection). In a future update, we will try to motivate what we're considering before jumping into the weeds. The author offers his apologies to the reader.}

Let $X := \Q_+^{n-1} \times \Q_+$ (so we think of elements as pairs $(\vec{u},x)$). There will be two core types of perturbations, `shifts' (which will additively alter the right coordinate in pair) and `dilations' (which can multiplicatively change either coordinate in a pair).

\subsubsection{Shifts}
We shall start with the additive perturbations. Given $\vec{u}\in \Q_+^{n-1}$, and a pair $(A,B)$ of disjoint sets $A,B\subset [n-1]$, we define the ratio $\rho_{(A,B)}(\vec{u}) := \frac{\prod_{a\in A} u_a}{\prod_{b\in B} u_b}$. 

Consider $\vec{\lambda} \in \Q_{\ge 0}^\Omega$, whose indices belong to some set $\Omega$ of pairs $(A,B)$ satisfying the aforementioned assumptions (namely, $A,B\subset [n-1]$ are disjoint). We now define the $\vec{\lambda}$-shift:
\[\sigma_{\vec{\lambda}}: (\vec{u};x) \mapsto (\vec{u};x'),\]where
\[x' := x+ \sum_{(A,B)\in \Omega} \lambda_{(A,B)}\rho_{(A,B)}(\vec{u}).\]

Let $\Sigma_\Omega:= \{\sigma_{\vec{\lambda}}:\vec{\lambda}\in \Q_{\ge}^\Omega\}$ denote the sets of shifts supported on coordinates in $\Omega$.

It is very clear that $\Sigma_\Omega$ is closed under composition, as \[\sigma_{\vec{\lambda}}\circ \sigma_{\vec{\lambda}'} = \sigma_{(\lambda_{(A,B)}+\lambda_{(A,B)}')_{(A,B)\in \Omega}}\]for any two vectors $\vec{\lambda},\vec{\lambda}'$.

\subsubsection{Dilations} Here, we will again let $\Omega$ be some set of pairs $(A,B)$. 

Now, given $q_* = (q_1,q_2)\in \Q_+$, we define the $q_*$-rescale
\[R_{q_*}: (\vec{u};x) \mapsto (q_1 \cdot \vec{u};q_2 x).\]Write $\QQ_*:=\{ R_{q_*}: q_*\in \Q_+^2\}$ for the total set of dilation operators.

Clearly, $\QQ_*$ is also closed under composition. We now will be interested in understanding how these rescalings compose with shifts. 

To do this, we require a new definition. Given two vectors $\vec{x},\vec{y}\in \Q_+^\Omega$, we now define $\vec{x}\odot \vec{y} := (x_{(A,B)}y_{(A,B)})_{(A,B)\in \Omega}$ to be their coordinate-wise product.

\begin{lem}\label{basic identity}
    For each $q\in \Q_+$, there exists $\Tilde{q}\in \Q_+^{\Omega}$ such that
    \[\vec{\rho}_{q\cdot \vec{u}}= \Tilde{q} \odot \vec{\rho}_{\vec{u}}\]for all $\vec{u}\in \Q_+^{n-1}$.
    \begin{proof}
        Explicitly, $\Tilde{q}= (q^{|A|-|B|})_{(A,B)\in \Omega}$. We leave verification of this to the reader, as it just follows from considering our definitions.
    \end{proof}
\end{lem}

\begin{cor}\label{commutator}
    Given $\vec{\lambda}\in \Q_{\ge 0}^\Omega$ and $q_* \in \Q_+^2$, there exists $\vec{\lambda}'\in \Q_{\ge 0}^\Omega$ so that \[R_{q^*}\circ \sigma_{\vec{\lambda}} = \sigma_{\vec{\lambda}'}\circ R_{q^*}.\]
    \begin{proof} Write $q_* = (q_1,q_2)$, whence we have $R_{q_*} = R_{(q_1,1)} \circ R_{(1,q_2)}$. Now, first note that
    \[R_{(1,q_2)}\circ \sigma_{\vec{\lambda}} = \sigma_{q_2\cdot \vec{\lambda}}  \circ R_{(1,q_2)}.\]Next, recalling the definition of $\Tilde{\cdot}$ from Lemma~\ref{basic identity}, and taking $\vec{\lambda}':= \Tilde{(1/q_1)}\odot q_2\cdot \vec{\lambda}$, one gets 
    \[ R_{(q_1,1)} \circ \sigma_{q_2\cdot \vec{\lambda}} = \sigma_{\vec{\lambda}'}\circ R_{(q_1,1)},\]
    completing the proof.    \end{proof}
\end{cor}

\subsubsection{Our set of perturbations}

Throughout this section, we fix some set $\Omega$ of pairs $(A,B)$, where $A,B\subset [n-1]$ are disjoint.

We now define $\PP_\Omega := \{\sigma \circ R: \sigma\in \Sigma_\Omega,\, R \in \QQ_*\}$. This will be the set of perturbations which we focus on.

To start, we establish that this set is closed under composition.
\begin{lem}\label{closed}
    Given any $\sigma^{(1)},\sigma^{(2)} \in \Sigma_\Omega, R^{(1)},R^{(2)}\in \QQ_*$, we can find a shift $\sigma\in\Sigma_\Omega$ and rescale $R\in \QQ_*$ so that
    \[\sigma^{(1)} \circ R^{(1)} \circ \sigma^{(2)} \circ R^{(2)} = \sigma \circ R.\]
    \begin{proof}
        Since $\Sigma_\Omega,\QQ_*$ are each closed under composition, it suffices to find $\sigma'\in \Sigma_\Omega, R' \in \QQ_*$ so that $R^{(1)}\circ \sigma^{(2)} = \sigma' \circ R'$. To do this, we simply appeal to Corollary~\ref{commutator}.
    \end{proof}
\end{lem}

\subsection{Connections to consistency}

\zh{TODO: motivate why these things arise ETC}

As in the last subsection, let $X = \Q_+^{n-1}\times \Q_+$. We now introduce several collections of colorings on $X$.
\begin{defn}
    Given any $S\subset [n-1]$, we we write $\CC_S$ to denote the set of finite colorings $C$ of $X$, which are of the form
    \[C((\vec{u};x)) = C_0(x\prod_{s\in S} u_s)\]for some coloring $C_0$ of the positive rationals.
\end{defn}

Let us state an equivalent version of Lemma~\ref{general term shift}, in terms of our new definitions.
\begin{lem}\label{general shift restated}
    Consider any $S\subset [n-1]$.

    Let $\Omega$ be the set of pairs $(A,B)$ with $A,B\subset [n-1]$ with $|A|+|S|>|B|$.

    Then, for every choice of $r<\infty$ and finite set $\Xi \subset \Q_{\ge 0}^\Omega$, there exists a finite set $P'\subset \PP_\Omega$, such that:

    Given any $x\in X:= \Q_+^{n-1}\times \Q_+$, and any $r$-coloring $C:X\to [r] \in \CC_S$, we can find $p'\in P'$ and define
    \[x' := p'(x)\]
    so that for each $\vec{\xi}\in \Xi$, we have
    \[C(\sigma_{\vec{\xi}}(x')) = C(x').\] 
\end{lem}\noindent This will be enough to establish Proposition~\ref{stable ext}. Indeed, let us state a `corrected' form of Delusion~\ref{dream statement}:
\begin{prp}\label{legit final goal}
    Let $\mathfrak{X}$ be a set of new $n$-families. Write $\Omega := \newp(\mathfrak{X})$, and suppose $|A|>|B|$ for each $(A,B)\in \Omega$.

    Then, for every $r<\infty$ and finite $P\subset \PP_\Omega$, there exists a finite set $P'\subset \PP_\Omega$ such that:

    Given any $x\in X$, along with $r$-colorings $C_\II:X \to [r] \in \bigcup_{S\subset [n-1]} \CC_S$ for each $\II \in \mathfrak{X}$, and choices\footnote{Here we can allow $\sigma_\II$ to be an arbitrary element of $\Sigma_\Omega$. We apologize for the unfortunate clash of notation with Subsection~\ref{motivation}.} $\sigma_\II\in \Sigma_\Omega $ for each $\II\in \mathfrak{X}$, then there exists $p' \in P'$, so that writing
    \[x' := p'(x),\]
    we have
    \[C_\II( \sigma_\II\circ p(x)) = C_\II(p(x))\]for each $\II\in \mathfrak{X},p\in P$.
\end{prp}
\noindent Some un-wrapping of definitions reveals that this indeed implies Proposition~\ref{stable ext} (currently we omit the details, although the broad strokes are discussed in Subsection~\ref{motivation}).

In our next (and final) section, we shall prove Proposition~\ref{legit final goal}, completing the goal of this paper. Our approach is rather general, showing that these sorts of results can be proven for any system $(X,\Sigma_\Omega,\PP_\Omega,\CC)$, as long as you have a result like Lemma~\ref{general shift restated}, along with some additional rather mild properties.

We conclude this section with proving one last mild property, which will be need later. 

\begin{lem}\label{compatible lemma}
    Consider any $S\subset [n-1]$.

    Let $\Omega$ be some set of pairs $(A,B)$ of disjoint $A,B\subset [n-1]$.

    For any $C:X\to [r]\in \CC_S$ and finite $H\subset \QQ_*$, we can define an auxiliary coloring $C':X\to [r^{|H|}]\in \CC_S$, where 
    \[C'(x^{(1)}) = C'(x^{(2)})\]if and only if $C(h(x^{(1)}) = C(h(x^{(2)}))$ for all $h\in H$.
    \begin{proof}
        Let $\pi= \pi_S:X\to \Q_+$ be the map $(\vec{u};x)\mapsto x\prod_{s\in S} u_s$. By definition, $\CC$ is the set of colorings of the form $C_0\circ \pi$ for arbitrary colorings $C_0$ of $\Q_+$.

        Let $H = \{h_1,\dots,h_m\}$. We wish to take $C'(x) := 1+\sum_{i=1}^m r^{i-1}(C(h_i(x))-1)$. It is easy to see that $C'$ would have the desired properties, so we just need to check that $C'\in \CC_S$. 
        
        For this, it suffices to verify that
        \[\pi(x^{(1)}) = \pi(x^{(2)})\implies C'(x^{(1)}) = C
        (x^{(2)})\] for all $x^{(1)},x^{(2)}\in X$. This follows quickly from the fact that $\pi(x^{(1)}) =\pi(x^{(2)})$ implies $\pi(h(x^{(1)})) = \pi(h(x^{(2)}))$ for all choicse of $x^{(1)},x^{(2)}\in X,h\in \QQ_*, S\subset [n-1]$.
    \end{proof}
\end{lem}
\begin{rmk}
    The property ``$\pi(x^{(1)}) =\pi(x^{(2)})$ implies $\pi(h(x^{(1)})) = \pi(h(x^{(2)}))$'' is a prime example of a map $h$ behaving `homogenously' with respect to a set of colorings $\CC$ (meaning that the colorings can reliably distinguish between `$x$' and `$h(x)$' for all points $x\in X$). This kind of notion will be informally mentioned several times in the upcoming section, and will be more thoroughly touched upon in a future draft.
\end{rmk}

\section{A refined multi-task Ramsey result}\label{finer multitask}

As we alluded to in Remark~\ref{allusion}, being compactly Ramsey is a rather ridiculous property. Already when $P$ contains a single map, things are tricky.
\begin{rmk}\label{too strong}
    Consider a map $p\in \PP$. Let $X'\subset X$ be the set of points $x$ where \[x,p(x), p\circ p(x), \dots \]are all distinct. It is not hard to define a $2$-coloring $C$ of $X$ where $C(x)\neq C(p(x))$ for all $x\in X'$ (this is essentially what is done in Remark~\ref{trivial N}). And we can also define a $3$-coloring of $X$ where $C(x) \neq C(p(x))$ for all $x\in X$ where $x\neq p(x)$.

    Thus, if $\CC$ contains all colorings of $X$, we already have that $P'\to_{\CC_3} \{p\}$ implies that: for each $x\in X$, there exists $p'\in P'$ where $p\circ p'(x) = p'(x)$ (meaning $p$ must have fixed points).
\end{rmk}\noindent For larger sets $P$, the restrictions only continue to grow.

Some of these difficulties are handled by the restrictions on the colorings we consider. However, due to the `homogeneous' nature of dilations (as were discussed in Remark~\ref{homogen discuss}), the framework of Section~\ref{digression} still cannot work. Our solution is to work with a weaker property.

\subsection{A dissection} Let us dissect what is really necessary in the proof of Proposition~\ref{compactness implies multitask}.

First, let's state a precise task:

\begin{defn}\label{instance definition}
    Suppose we are given maps $p_1,\dots, p_\ell \in \PP$, along with sets of colorings $\CC^{(1)},\dots, \CC^{(\ell)}$. We shall say that $((p_i,\CC^{(i)}))_{i\in [\ell]}$ is a \textit{multi-task Ramsey instance}, if for every choice of $(C_i\in \CC^{(i)})_{i\in [\ell]}$, there exists some $x\in X$ where
    \[C_i(p_i(x)) = C_i(x)\]for each $i\in [\ell]$. 
\end{defn}

Now suppose one wants to prove that some $((p_i,\CC^{(i)}))_{i\in [\ell]}$ is a multi-task Ramsey instance. By mimicing our proof of Proposition~\ref{compactness implies multitask}, it would suffice to find sets $P_1',\dots,P_\ell' \subset \PP$, where, for each $t\in [\ell]$, $C_t\in \CC^{(t)}$, there exists some $p_t'\in P_t'$ where
\[C_t(p_t \circ p_{new} \circ p_t' (x)) = C_t(p_{new}\circ p_t'(x))\]for each $p_{new} \in P_{new} := P_{t+1}' \circ \dots P_{\ell}'$.

Let us state this more formally. 

\begin{defn}
    Given $p\in \PP$ and $H\subset \PP$, and a coloring $C$ of $X$, we say that $x$ is $H$-stably $p$-consistent if 
    \[C(p\circ h(x)) = C(h(x))\]for all $h\in H$.
\end{defn}
\begin{defn}
    Given a set of colorings $\CC$, along with $H\subset \PP$ and $p\in \PP$, we say that $P'$ is a $(H,p,\CC)$-stabilizer if for every $C\in \CC$ and $x\in X$, we can find $p'\in P'$ where $p'(x)$ is $H$-stably $p$-consistent under $C$.
\end{defn}

\begin{defn}
    Given $p\in \PP$ and a set of colorings $\CC$, we say that $(p,\CC)$ is \textit{compactly-stabilizable} if for every finite $H\subset \PP$ and $r<\infty$, we can find a finite $P'\subset \PP$ which is a $(H,p,\CC_{\le r})$-stabilizer. 
\end{defn}

The sufficient condition from before tells us that:
\begin{prp}\label{comp stab to multistab}
    Suppose $((p_i,\CC^{(i)}))_{i\in [\ell]}$ is a finite tuple of pairs $(p_i,\CC^{(i)})$ which are each compactly-stabilizable.

    Then, for each finite $H\subset \PP$, there exists some finite $P'\subset \PP$, such that for each $x\in X$ and choice of $C_i\in \CC^{(i)}$ for $i\in [\ell]$, we can find $p'\in P'$ so that, writing $x' := p'(x)$, we have:
    \[C_i(p_i\circ h(x')) = C_i(h(x'))  \]for all $h\in H$.
    \begin{proof}
        One simply inducts, doing ``backtracking'' as in Proposition~\ref{compactness implies multitask}. 

        Suppose we are given some finite $H\subset \PP$, along with $r<\infty$. Initialize with $H_t := H$.
        
        First, since $(p_\ell,\CC^{(\ell)})$ is compactly-stabilizable (and $H_\ell\subset \PP$ is finite), we can define a finite $P'_\ell = P'_\ell((p_\ell,\CC^{(\ell)}),r,H_\ell)$ so that for each choice of $C_\ell\in \CC_{\le r}^{(\ell)}$ and $x\in X$, there exists $p_\ell'\in P_\ell'$, where
        \[C_\ell(p_\ell\circ h \circ p_\ell'(x)) = C_\ell(h\circ p_\ell'(x))\]for all $h\in H_\ell$. 
        
        We then define $H_{\ell-1} := H_\ell \circ P_\ell'$. Clearly $H_{\ell-1}$ is finite, as $|H_{\ell-1}|\le |H_\ell||P_\ell'|$.

        Now for each time $t =\ell-1,\dots, 1$, we simply repeat the above step. We consider the finite set $H_t\subset \PP$ defined previously. Due to $(p_t,\CC^{(t)})$ being compactly-stabilizable (and the finite-ness of $H_t$), we can find some finite $P_t' =P'_t((p_t,\CC^{(t)}),r,H_t)\subset \PP$ where
        for each $C_t\in \CC_{\le r}^{(t)},x\in X$, we can pick $p_t'\in P_t'$ so that
        \[C_t(p_t\circ h_t\circ p_t'(x)) = C_t(h_t\circ p_t'(x))\]for all $h_t\in H_t$. Then we continue by setting $H_{t-1} := H_t\circ P_t'\subset \PP$ (which is again finite). 

        After running this up to time $t=1$. We simply take $P' := P_\ell'\circ P_{\ell-1}'\circ \dots \circ P_1'$ (which is again finite). This will have the desired properties.
    \end{proof}
\end{prp}

So, to end up proving Proposition~\ref{stable ext}, one only needs a way to show that pairs $(p,\CC)$ are compactly-stabilizable. We do this in the next subsection, completing the proof.

\subsection{Stabilizing}

As we have often doen before, we consider a ground set $X$ and a set of perturbations $\PP$. But now, we also introduce another set of perturbations $\HH$ (which we think of as being `homogeneous transformations', which are transformations that colorings can react to). 

\thide{A prime example to keep in mind, is that often $\CC$ is taken to be the set of colorings $C$ on $X$, of the form $C_0\circ \pi$}

\subsubsection{Commuting nicely} We isolate some ideas which are somewhat convenient (to ensure technical things like closure), but is ultimately not at the heart of the argument.

Now, given $X,\PP,\HH$, we say that $(\PP,\HH)$ \textit{commute nicely} \zh{maybe more apt name is commute normally or something? but I don't know group theory...}if for every $h\in \HH,p\in \PP$, there exists $p'\in \PP,h'\in \HH$ so that
\[h\circ p = p'\circ h'.\]

\begin{rmk}
    Here is an example. If we take $X = \N$, and
    \[\PP : = \{n\mapsto n+m:m\in \N\},\]
    \[\HH := \{n\mapsto dn: d\in \N\},\]
    then $(\PP,\HH)$ commute nicely. Note however, that $(\HH,\PP)$ do \textit{not} commute nicely here (since the map $n\mapsto 2n+1$ cannot be be written as $h'\circ p'$ for some $h'\in \HH,p'\in \PP$).\zh{add example/exercise about representing vdw in this framework}
\end{rmk}

We note some basic facts about this definition.
\begin{prp}
    Suppose $(\PP,\HH)$ commute nicely. Then $\PP\circ \HH := \{p\circ h:p\in \PP,\,h\in \HH\}$ is closed under composition.
    \begin{proof}
        For $p_1,p_2 \in \PP,h_1,h_2\in \HH$, we can find $p'\in \PP$ and $h'\in \HH$ so that
        \[h_1\circ p_2 = p'\circ h',\]whence
        \[(p_1\circ h_1)\circ (p_2\circ h_2) = (p_1\circ p')\circ (h'\circ h_2)\in \PP\circ \HH.\]
    \end{proof}
\end{prp}

\begin{prp}\label{covering}
    Given finite $P\subset \PP$ and finite $T\subset \PP\circ \HH$, there exist finite $P'\subset \PP$ and $H'\subset \HH$ so that
    \[P\circ T\subset P'\circ H'\]
    (meaning for each $p \in P, t\in T$, there exists $p'\in P',h'\in H'$ where $p\circ t = p'\circ h'$).
    \begin{proof}
        Obvious.
    \end{proof}
\end{prp}

\subsubsection{The important properties}

Similarly, we say $(\PP,\HH)$ \textit{uncommutes very-nicely}\footnote{The author apologizes for the atrocious terminology here.}, if for every $p\in \PP,h\in \HH$, there exists $p'\in \PP$ so that
\[p\circ h = h\circ p'.\]

We now introduce auxiliary colorings.
\begin{defn}
    Given an $r$-coloring $C:X\to [r]$, and a finite $H\subset \HH$, we can define an $r^{|H|}$-coloring $C'=C'[C,H]: X \to [r^{|H|}]$, where
    \[C'(x_1) = C'(x_2)\]if and only if
    \[C(h(x_1)) = C(h(x_2))\]for all $h\in H$. Indeed, one simply fixes some bijection $\iota$ between the set of colorings $\{c:H\to [r]\}$ and $[r^{|H|}]$, and then can define $C'(x):= \iota (h\mapsto C(h(x))) $. Going forward, we shall assume that some arbitrary choice of $C'$ is always specified.\zh{TODO: clarify how we indentify the set of maps $H\to [r]$ with $[r^{|H|}]$ or whatever.}
\end{defn}

\begin{defn}
    Given a collection of colorings $\CC$, and a collection of maps $\HH$, we say that $\CC$ is $\HH$-compatible if: For each $C\in \CC$, and every finite $H\subset \HH$, we have that $C'[C,H]\in \CC$.
\end{defn}

\subsubsection{Putting the tools together}
We can now show how these basic properties give us a rather general sufficiency criterion for being compactly-stabilizable.
\begin{prp}\label{sufficiency for compact stab}
    Consider a triple $(X,\TT,\CC)$, where $\TT = \PP\circ \HH$ for some choice of perturbations $\PP,\HH$ on $X$. 
    
    Suppose that:
    \begin{itemize}
        \item $(\PP,\HH)$ commutes nicely;
        \item $(\PP,\HH)$ uncommutes very-nicely;
        \item $\CC$ is $\HH$-compatible;
        \item For every $r<\infty$ and finite $P\subset \PP$, there exists a finite $T\subset \TT$ so that\footnote{Recall that given a set of colorings $\CC$, we write $\CC_{\le r}$ to denote the subset of $C$ where $C(X)\subset [r]$.} $T\to_{\CC_{\le r}} P$ (in the notation of Defintion~\ref{compact definition}).
    \end{itemize}
    Then for every finite $P\subset \PP,T\subset \TT$ and $r<\infty$, we can define a finite $T'\subset \TT$ so that, for each $C\in \CC_{\le r}$ and $x\in X$, we can find $t'\in T'$ where
    \[C(p\circ t\circ t'(x)) = C(t\circ t'(x))\]for all $p\in P,t\in T$.
    
    In particular, for each $p\in \PP$, we have that $(p,\CC)$ is compactly-stabilizable (wrt to the set of perturbations $\TT$).
    \begin{proof}
        Fix any finite $T\subset \TT$ and $r<\infty$. 
        
        We can define maps $\pi_1:\TT\to\PP,\pi_2:\TT \to \HH$, so that $\pi_1(t) \circ \pi_2(t) = t$ for each $t\in \TT$. With foresight, write $H^* := \pi_2(T), P^* := \pi_1(T) \cup (P\circ \pi_1(T))$. By our second bullet, we can define a finite $P_{new}\subset \PP$, so that for each $p^*\in P^*, h^*\in H^*$, there exists $p_{new} \in P_{new}$ where
        \[p^*\circ h^* = h^* \circ p_{new}.\]
        
        \hide{The relevent fact about the above definition is: if $\{h^*\circ p_{new}:(x)) = C(h^*(x))$ some coloring $C$, then for each $p\in P,t\in T$, we will have
        \[C(p\circ t(x))= C( = C(t(x))\]}

        \hide{
        By Lemma~\ref{covering} (and the assumption of our first bullet), we can define finite sets $P^*\subset \PP$ and $H^*\subset \HH$ so that \[ T\cup P\circ T\subset P^*\circ H^*.\]
        Then by the assumption of the second bullet, we can find a finite $P_{newer}\subset \PP$ so that for each $p\in P$ and $t= p_{new} \circ h_{new}$, there exists $p_{newer}\in P_{newer}$ where
        \[p\circ p_{new}\circ h_{new} = h_{new} \circ p_{newer}.\]}

        Let $r_{new} := r^{|H^*|}$. By the assumption of our third bullet, there exists some finite $T'\subset\TT$ so that for each $C\in \CC_{\le r_{new}},x\in X$, we can find $t'\in T'$ where
        \[C(p_{new,1}\circ t'(x)) = C(p_{new,2}\circ t'(x))\]for all $p_{new,1},p_{new,2}\in P_{new}$.

        We claim this suffices. Indeed (by our third bullet), for any $C\in \CC_{\le r}$, the auxiliary coloring $C' = C'[C,H_{new}]$ belongs to $\CC_{\le r_{new}}$. Thus there exists some $t'\in T'$ where, writing $x' := t'(x)$, we get 
        \[C'(p_{new,1}(x')) = C'(p_{new,2}(x'))\]for all $p_{new,1},p_{new,2}\in P_{new}$. By definition of $C'$, this means that
        \[C(h^*\circ p_{new,1}(x')) = C(h^*\circ p_{new,2}(x'))\]for all $p_{new,1},p_{new,2}\in P_{new}$ and $h^*\in H^*$.
        
        In turn (by our choice of $P_{new}$), this implies that
        \[C(p_1^*\circ h^*(x'))= = C(p_2^*\circ h^*(x'))\]for each $p_1^*,p_2^*\in P^*$ and $h^*\in H^*$. We are then done, noting that for each $p\in P,t\in T$, we can find $p_1^*,p_2^*\in P^*$ and $h^*\in H^*$ so that
        \[p\circ t = p_1^*\circ h^*,\,t = p_2^*\circ h^* \](this is simply by definition of $H^*,P^*$, we leave verification to the reader in this version).
    \end{proof} 
\end{prp}

Thus, in order to establish Proposition~\ref{legit final goal}, we just need to confirm that the properties above hold.
\begin{proof}[Proof of Proposition~\ref{legit final goal}]Fix some $S\subset [n-1]$ and set of pairs $\Omega$ where $|A|>|B|$ for each $(A,B)\in \Omega$.

Then define $X:=\Q_+^{n-1}\times \Q_+$ and $\Sigma_\Omega,\QQ_*,\PP_\Omega,\CC_S$ as in the last section.

We shall use Proposition~\ref{sufficiency for compact stab} to deduce that $(\sigma,\CC_S)$ is compactly-stabilizable for each $\sigma\in \Sigma_\Omega$. It will then follow that $(\sigma,\bigcup_{S\subset [n-1]} \CC_S)$ is compactly-stabilizable, by repeatedly applying the following observation: 
\begin{clm}
        Let $X$ be an arbitrary space, $\PP$ be an arbitrary collection of compositions, and $\CC^{(1)},\CC^{(2)}$ be arbitrary sets of colorings.
        
        Suppose that $(\sigma,\CC^{(1)}),(\sigma,\CC^{(2)})$ are both compactly-synchronizable. Then $(\sigma,\CC^{(1)}\cup \CC^{(2)})$ is also compactly-synchronizble.
        \begin{proof}
            Given any finite $P\subset \PP$ and $r<\infty$, we can find finite $P'_1,P'_2\subset \PP$, where for all $x\in X,C_1\in \CC^{(1)}_r,C_2\in \CC^{(2)}_r$, we can find $p_1'\in P_1',p_2'\in P_2'$ where
            \[C_1(\sigma \circ p_1(x)) = C_1(p_1(x))\quad \text{and}\quad C_2(\sigma\circ p_2(x)) = C_2(p_2(x)).\]
                
            Thus taking $P' := P'_1\cup P'_2$ will have the desired properties.
            \end{proof}
        \end{clm}
    Now, assuming $(\sigma,\bigcup_{S\subset [n-1]} \CC_S)$ is compactly-stabilizable for each $\sigma \in \Sigma_\Omega$, then we can immediately apply Proposition~\ref{comp stab to multistab} to deduce Proposition~\ref{legit final goal}, as desired.

    So, it remains to establish $(\sigma,\CC_S)$ is compactly-stabilizable. First note that we can write $\PP_\Omega = \Sigma_\Omega \circ \QQ_*$ (indeed, this is the content of Lemma~\ref{closed}. Now, we shall invoke Proposition~\ref{sufficiency for compact stab} with $\TT:= \PP_\Omega,\PP:= \Sigma_\Omega,\HH:= \QQ_*$. We must verify the four bullets in the statement of said proposition.
    
    Now Corollary~\ref{commutator} proves that $(\Sigma_\Omega,\QQ_*)$ commutes nicely, and a practically identical argument shows that $(\Sigma_\Omega,\QQ_*)$ uncommutes very-nicely. So this handles our first two bullets.

    Lemma~\ref{compatible lemma} tells us that $\CC_S$ is $\HH$-compatible, handling bullet three. Finally, Lemma~\ref{general shift restated} implies the fourth bullet.

    This completes our proof (finally). 
\end{proof}

\begin{rmk}\label{inspection note}
    Inspecting the proof above, we see that we've proven that given any $S\subset [n-1]$ and set of pairs $\Omega$ where $|A|+|S|>|B|$ for each $(A,B)\in \Omega$, that $(\sigma,\CC_S)$ is compactly-stabilizable (wrt to $\PP_\Omega$).
\end{rmk}

\section{Sketch of new theorem}\label{sketch of new}

In this section, we briefly show how to modify the work from before to prove Theorem~\ref{rational main}. By the reductions in Section~\ref{family disjoint union}, it suffices to establish Proposition~\ref{rat build v}.

To do this, we wish to prove the following version of the stable extension result (Proposition~\ref{stable ext}):
\begin{lem}\label{full extension}
    Let $\mathfrak{X}$ be the collection of all new $n$-family. Write $\Omega := \newp(\mathfrak{X})$.
    
    Then for every $r<\infty$ and finite $Q\subset \Q_+$, there exists some finite $Q'\subset \Q_+$, so that for every $C:\Q_+\to [r]$ and $\vec{u}\in \Q_+^{n-1}$, there exists some $q'\in Q'$ and $x'\in \Q_+$ so that defining:
    \[\vec{v}:= (q'\cdot u_1,\dots,q'\cdot u_{n-1},x'),\]we have that $q\cdot \vec{v}$ is $\mathfrak{X}$-consistent for all $q\in Q$.
\end{lem}
\noindent (here we have chosen to omit all the unnecessary information from the statement of Proposition~\ref{stable ext} which was highlighted Remark~\ref{unnecessary}; we can prove the same strength of result in this general setting) By doing the exact same induction argument as in Subsection~\ref{the induction}, Lemma~\ref{full extension} will easily imply Proposition~\ref{rat build v}.

So, here is how to prove Lemma~\ref{full extension}:

Let $\mathfrak{X}^*$ denote the collection of new $n$-families $\II$ with $|\II| > 1$. It suffices to attain $\mathfrak{X}^*$-consistency, since we can ignore familes where $|\II| = 1$ (as then $\varphi_\II = \varphi_{\{f(\II)\}}$ trivially).   
The idea is to now order families $\II_1,\dots,\II_\ell$ belonging to $\mathfrak{X}^*$, in increasing order by $|f(\II)|$. For $t =1,\dots,\ell$, define \[\Omega_t:= \newp(\{\II_i: i\in \{t,t+1,\dots,\ell\})\] and \[S_t:= f(\II_t)\setminus \{n\}.\]One can check that for $(A,B)\in \Omega_t$, we'll always have that $|A|+|S_t|> |B|$ (since $A$ will never be empty (for our collection $\mathfrak{X}^*$ of $\II$)). Thus, as noted by Remark~\ref{inspection note}, we will have that $(\sigma, \CC_{S_t})$ will be compactly-stabilizable inside $\PP_{\Omega_t}$ for each $t$. 

Then you can simply mimic the proof of Lemma~\ref{comp stab to multistab}, using the fact that $\Omega_t\supset \Omega_{t+1}$ for all $t$. Nothing really needs to be changed, you just need to consider the order/nesting of things.

\hide{\section{Conclusion}

\subsection{The value of polynomial van der Waerden}

We feel obligated to recall a conjecture of Alweiss. See Conjecture~4.1 of: \url{https://arxiv.org/pdf/2211.00766.pdf} (TODO: write this in own words).}

\end{document}